\documentclass[reqno,  twoside, a4paper, 11pt]{amsart}
\usepackage[english]{babel}
\usepackage{amsmath}
\usepackage[latin1]{inputenc}
\usepackage[colorlinks=true,draft=false]{hyperref}
\usepackage{amsfonts}
\usepackage{thmtools}
\usepackage{amssymb}
\usepackage[pagewise,mathlines]{lineno}
\usepackage{tikz-cd}
\usepackage{amsthm}
\usepackage{fixme}
\usepackage{url}
\usepackage{microtype}
\usepackage{mathrsfs}
\usepackage[all]{xy}
\usepackage{rotating}
\usepackage{enumitem}

\usepackage[nameinlink]{cleveref}

\newcount\hh
\newcount\mm
\mm=\time
\hh=\time
\divide\hh by 60
\divide\mm by 60
\multiply\mm by 60
\mm=-\mm
\advance\mm by \time
\def\hhmm{\number\hh:\ifnum\mm<10{}0\fi\number\mm}

\DeclareMathOperator{\sep}{sep}

\DeclareMathOperator{\insep}{insep}

\DeclareMathOperator{\Swan}{Swan}

\DeclareMathOperator{\Aut}{Aut}

\DeclareMathOperator{\Stab}{Stab}

\DeclareMathOperator{\Gal}{Gal}

\DeclareMathOperator{\norm}{norm}

\DeclareMathOperator{\GL}{GL}

\DeclareMathOperator{\red}{red}

\newcommand{\Z}{\mathbb{Z}}

\newcommand{\N}{\mathbb{N}}

\newcommand{\F}{\mathbb{F}}
\renewcommand{\P}{\mathbb{P}}

\renewcommand{\subset}{\subseteq}

\renewcommand{\phi}{\varphi}

\declaretheoremstyle[spaceabove=15pt,spacebelow=15pt, qed=$\square$]{defstyle}
\declaretheoremstyle[
	spaceabove=15pt,
	spacebelow=15pt, 
	bodyfont=\itshape,
	qed=$\square$]{thmstyle}
\declaretheorem[style=thmstyle]{Theorem}
\declaretheorem[sibling=Theorem, style=thmstyle]{Proposition}
\declaretheorem[sibling=Theorem,  style=thmstyle]{Lemma}
\declaretheorem[sibling=Theorem, style=thmstyle]{Corollary}

\declaretheorem[sibling=Theorem, style=defstyle]{Definition}
\declaretheorem[sibling=Theorem, style=defstyle]{Remark}

\declaretheorem[style=thmstyle]{Claim}
\setenumerate{label=(\alph*), ref=(\alph*)}

\begin{document}
\title{A note on fierce ramification}

\author{H\'el\`ene Esnault}
\address{
Freie Universit\"at Berlin, Arnimallee 3, 14195, Berlin,  Germany}
\email{esnault@math.fu-berlin.de}
\author{Lars Kindler}
\email{ lars.kindler@posteo.de }
\author{Vasudevan Srinivas}
\address{School of Mathematics, Tata Institute of Fundamental Research, Homi
Bhabha Road, Colaba, Mumbai-400005, India} 
\email{srinivas@math.tifr.res.in}
\date{}
\maketitle
\section{Introduction}
Let $k$ be an algebraically closed field of characteristic $p>0$ and let $\ell$ 
be a prime different from $p$. One fixes an algebraic closure 
$\bar \F_\ell$ of $\F_\ell$. Let $\bar{C}$ be a nonsingular connected projective 
curve over $k$, with a dense Zariski 
open subset $C\hookrightarrow \bar{C}$,  and a geometric point $\bar c\in C$.  
To   any  continuous representation $\rho:\pi_1(C,\bar{c})\to 
\GL_r(\bar{\F}_\ell)$,  thus with values 
$\GL_r(\F_{\ell^n})$ for some non-zero natural number $n$, 
one associates  its Swan  conductor $Sw(\rho)$ in the group $Z_0(\bar C)$ is 
zero-cycles.  It is an effective divisor 
 supported on $\bar{C}\setminus C$, which  measures the wild ramification of 
$\rho$ (\cite[III.20]{Ser67}, \cite[1.1.2]{Lau81}, \cite[4.4]{KR14}.)

The definition of the Swan conductor is local: if $K$ is the function 
field of $C$, then for each closed point $x\in \bar{C}\setminus C$, consider the 
corresponding complete discrete valuation field obtained from $K$ by
completion at $x$, and the associated local Galois 
representation $\rho_x$ into $\GL_r(\bar{\F}_{\ell})$ (well-defined up to conjugation). 
The non-negative integer invariant assigned to it, see e.g. \cite[4.82,4.84]{KR14},  is the 
coefficient of  $x$ in the zero cycle $Sw(\rho)$. We will use the term ``Swan conductor'' 
to mean either the local or global invariant, depending on the context. 

When $X$ is a normal connected variety of finite type over $k$, a modulus 
condition for continuous representations $\rho:\pi_1(X,\bar{x})\to 
GL_r(\bar{\F}_\ell)$ is defined in 
 \cite[Defn.~3.6]{EK12}: if $ X\hookrightarrow \bar{X}$ is a normal 
compactification of $X$, and $\Delta$ is a Cartier divisor supported in $\bar 
X\setminus X$, 
we say that  
$\rho:\pi_1(X,\bar{x})\to GL_r(\bar{\F}_\ell)$ has {\em  ramification bounded by 
$\Delta$},  if for any morphism $\varphi:\bar {C}\to \bar {X}$ from a  connected 
nonsingular projective curve $\bar{C}$, such that $\varphi^{-1}(X)=C$ 
is nonempty,
 the induced  representation  $\varphi_*\circ\rho:\pi_1(C,\bar{c})\to   
GL_r(\bar{\F}_\ell)$ verifies 
 $$Sw(\varphi_*\circ\rho)\preceq \varphi^*\Delta $$
 with respect to the order on $Z_0(\bar C)$. 
  As $\bar X$ is normal,  the intersection of its smooth locus $ \bar X_{\rm 
reg}$  with $\bar X\setminus X$ is dense, so $\Delta\cap 
\bar X_{\rm reg}$  as a divisor is a sum $\sum_i  m_i \Delta_i$,  where 
$\Delta_i$ is an irreducible divisor and $m_i\in \N$  are the called the 
multiplicities of $\Delta$.   Let $N$ be a natural number. We say that  {\it 
$\rho$  has
ramification bounded by $N$  if it has ramification bounded by $\Delta$ for an 
effective divisor supported in $\bar X\setminus X$ with multiplicities $m_i\le 
N$ for all $i$. } 
Similarly we say that {\it $\rho$ has ramification bounded 
by $N$ along a divisorial discrete valuation $v$} if there exists some normal compactification 
$X\hookrightarrow \overline{X}$ as above, with a boundary component divisor $D_0$ 
corresponding to $v$, such that $\rho$ has ramification bounded by an effective 
Cartier divisor $\Delta$, where the coefficient of $D_0$ in $\Delta$ is $\leq N$.

A natural question is whether, for a fixed effective Cartier divisor $\Delta$ as 
above, or for a chosen divisorial valuation $v$, the class of representations 
$\rho:\pi_1(X,\bar{x})\to GL_r(\bar \F_\ell)$ with ramification bounded by $\Delta$ 
(or with ramification bounded by $N$ along $v$)  has  other ``finiteness properties'' with 
respect to wild ramification. 

One such is the notion of fierce ramification along an irreducible component $D_0$ of 
$ \bar{X} \setminus X  $.  Let $\pi: Y\to X$ be the Galois cover of Galois group 
$ {\rm Im} (\rho)$ determined by the quotient  $\pi_1(X, \bar x)\to {\rm Im} 
(\rho)$, and $ \bar \pi: \bar Y\to \bar X$ be the normalization of $\bar X$ in 
the field of functions of $Y$.  So the smooth locus $\bar Y_{\rm reg} \subset \bar Y$  has 
complement of codimension at least $2$. 
Let $E_0$ be an irreducible component of $ \bar  \pi^{-1}(D_0 ) \cap \bar Y_{\rm 
reg}$.  Then the {\em fierce ramification index} of  
 $D_0$ is the purely inseparable degree of the function field extension 
$k(D_0)\subset k(E_0)$. It  depends on the local system defined by $\rho$ and 
$D_0$, not on the choices of $\bar x$ and $E_0$. 
Indeed,  ${\rm Im}(\rho)$ acts  transitively on the set $\{ E'_0\}$ of 
components of $\bar \pi^{-1}(D_0) \cap \bar Y_{\rm reg}$, on the set $\{ k(D_0) 
\hookrightarrow k(E'_0)\} $ of  extensions of $k(D_0) $ preserving the separable 
closures and the purely inseparable ones. Changing $\bar x$ conjugates the 
representation. The conjugation sends $\{ k(D_0) \hookrightarrow k(E'_0)\} $ 
defined for $\bar x$ to the corresponding set defined for the other base point, 
preserving the separable closures and the purely inseparable ones.  We say that  
$(\rho, X\hookrightarrow \bar X)$ has {\it  fierce ramification bounded by a 
natural number $M$}  if for all $D_0$,  the fierce ramification index of $D_0$ is at most $M$.

Similarly we have the notion of {\it  fierce ramification index along a divisorial valuation} $v$, 
which equals the fierce ramification index along $D_0$ for any normal compactification 
$X\hookrightarrow \bar{X}$ with a boundary component $D_0$ associated to the divisorial
valuation $v$. This notion depends only on the discrete valuation, since it can be defined 
using the extension of discrete valuation rings associated to $E_0\to D_0$.

The aim of this note is to prove that bounding the ramification  along a divisorial valuation 
$v$ also bounds the fierce ramification index.

\begin{Theorem}\label{thm:main} Let  $(X,\ \ell,\  v,\ r)$ 
be as above. Let $N$ be a natural number.  Then there is a natural number $M$  
such that for all  continuous representations $\rho:\pi_1(X,\bar{x})\to \GL_r(\bar{\F}_\ell)$
 of ramification bounded by $N$ along $v$, the fierce ramification of $(\rho, 
X\hookrightarrow \bar X)$ along $v$ is bounded by $M$. 
\end{Theorem}

The theorem answer positively a question posed by Pierre Deligne in 
\cite{Del16}.

The proof consists of two parts. In Section~\ref{sec:loc}, we first make a 
local analysis of ramification at a point on a nonsingular curve, in relation to 
a bound on the Swan conductor at 
that point. This is formulated in terms of a boundedness assertion for 
representations of the corresponding local Galois group (see 
Proposition~\ref{prop:boundedness}). 
To globalize the argument and perform the proof of Theorem~\ref{thm:main}, we 
notably use a version of the Bertini theorem  controlling that the inverse image 
of a curve by a ramified Galois cover remains unibranch \cite{Zha95}.  \\[.2cm]
{\it Acknowledgements}:  The first author thanks Pierre Deligne for his letter 
\cite{Del16} sent in relation to \cite{Esn17}. 
This note, which answers one of the two questions in his letter, has been 
circulating among experts for two years. The second question, asking whether 
fixing  $k$ algebraically closed, and $(X\hookrightarrow \bar X, \ell, r, 
\Delta)$ one can find a curve $C\hookrightarrow X$ such that the restriction of 
any irreducible representation of rank $\le r$ and ramification bounded by 
$\Delta$ remains irreducible, remains unanswered.  By Deligne's theorem 
\cite[Thm.~1.1]{EK12} and the standard Lefschetz theorem over finite fields (see 
\cite{Esnt17} and references in there), this is true over $k=\bar \F_p$ for the 
representations which descend to $\F_{p^m}$ for a  fix natural number $m$.

	\section{Local arguments} \label{sec:loc}
Let $k$ be an algebraically closed field of characteristic $p>0$ and let $K$ be 
a discretely valued, complete field  of characteristic $p$ with residue field 
$k$.  Fix algebraic an algebraic closure $\bar{K}$ and write 
$G_K:=\Gal(\bar{K}/K)$.

\begin{Proposition}\label{prop:boundedness}
Fix a prime $\ell\neq p$, and positive natural numbers $r,N$. There exists a
number $M(\ell,r,N)$ with the following property. For any  continuous
representation $\rho:G_K\rightarrow \GL_r(\F_{\ell^n})\subset 
\GL_r(\bar{\F}_\ell)$
with $Sw(\rho)$ bounded by $N$, there exists a finite Galois extension $L_\rho$ 
of $K$
of degree $\leq M(\ell, r, N)$, such that $\rho|_{L_\rho}:=\rho|_{G_{L_\rho}}$ is
tamely ramified.
\end{Proposition}
\begin{proof}
	Given $\rho$ as in the statement, we
	write $\bar{I}:=\rho(G_K)$,  and we let $\bar{P}$ denote the (unique) 
$p$-Sylow subgroup of $\bar{I}$. The Galois theory of 
	discretely valued fields shows that there is a short exact sequence
	\begin{equation}\label{extension}1\rightarrow \bar{P}\rightarrow 
\bar{I}\rightarrow\Z/M\rightarrow 1,\end{equation}
	where $(p,M)=1$. Writing $M=\ell^nM'$ with $(\ell,M')=1$, we have 
$\Z/M\cong \Z/\ell^n\times \Z/M'$.
	\begin{Claim}
	The exponent $n$ is bounded by a constant only depending on $r$.
\end{Claim}	
\begin{proof}In $\bar{I}$ there exists an element $\sigma$ of order precisely 
$\ell^n$ (e.g.~by the theorem of Schur-Zassenhaus). Considering $\sigma$ 
as an element of $\GL_r(\bar{\F}_\ell)$, we can write $\sigma=1+x$, with $x$ a 
nilpotent matrix such that $x^{\ell^{n-1}}\neq 0$ but $x^{\ell^n}=0$. This 
shows that $\ell^{n-1}\leq r$.
\end{proof}

\begin{Claim} The order of $\bar{P}$ is bounded by a constant only depending on 
$r,N$.\end{Claim}
\begin{proof}
Let $\tilde{\bar{I}}\subset \bar{I}$ be the preimage of $\Z/M'\subset \Z/M$. 
We proceed in several steps.
\begin{enumerate}
	\item As $(M',p)=1$, $\bar{P}\subset \tilde{\bar{I}}$ and $\bar{P}$ is 
the unique $p$-Sylow subgroup of $\tilde{\bar{I}}$. 
	\item As $(M',\ell)=1$, we see that $(\ell,|\tilde{\bar{I}}|)=1$.
	Recall Jordan's  theorem according to which  there is a constant $J(r)$,  the {\it Jordan constant}, depending only on 
$r$, such that, given a subgroup $\GL_r(\bar{\F}_{\ell})$ of order prime to $\ell$, it contains an abelian normal subgroup of index at most $J(r)$ (the constant is the same as for subgroups of $\GL_r({\mathbb C})$). It yields 
  	 a normal abelian subgroup 
	$A\lhd \tilde{\bar{I}}$, such that $B:=\tilde{\bar{I}}/A$ is of order  bounded by $J(r)$. 
	 Consider the following diagram:
\begin{equation*}
	\begin{tikzcd}
		1\rar&\dar[hookrightarrow] A\cap 
\bar{P}\rar&\dar[hookrightarrow]\bar{P}\rar&\dar[hookrightarrow]\bar{P}/A\cap 
\bar{P}\rar &1\\
		1\rar&A\rar&\tilde{\bar{I}}\rar&B\rar &1
	\end{tikzcd}
\end{equation*}
To prove the claim, it suffices to show that the order of $A\cap \bar{P}$ is 
bounded by a constant only depending on $N, r$. 
\item Translating the group theoretic picture via Galois theory, we obtain the 
following diagram of field extensions, where the arrows are labeled 
by the corresponding Galois groups.
\begin{equation*}
	\begin{tikzcd}
		L\\
		L'\uar{A}\\
		\tilde{L}\uar{B}\ar[bend right=45,swap]{uu}{\tilde{\bar{I}}}\\
		K\uar[swap]{\Z/\ell^n}\ar[bend left=45]{uuu}{\bar{I}}
	\end{tikzcd}
\end{equation*}
\item If $\rho_A:\rho^{-1}(A)\rightarrow \GL_r(\bar{\F}_{\ell})$ is the induced 
representation, then there exists a bound $N'$ depending only on $N,r$, 
such that $Sw(\rho_A)$ is bounded by $N'$. 


To see this, apply \Cref{lemma:pullback} below to $L'/K$ and $\rho_A$.
This yields the desired estimate with $N'\leq \ell^r J(r) N$.

%

\item For every element $\sigma\in A\cap \bar{P}$, we have $\sigma^{p^{N'}}=1$. 
Indeed, if $x_1,\ldots, x_t\in \Z_{>0}$ are the jumps of the ramification 
filtration of $A=\Gal(L/L')$ in ascending order, and if $x_0=0$, then
	
\[\Swan(\rho_A)=\sum_{i=1}^tx_i\underbrace{\dim(V^{A^{(x_{i})}}/V^{A^{(x_{i-1})}
})}_{\geq 1}\leq N',\]
	where $V$ is the $r$-dimensional $\bar{\F}_{\ell}$-vector space on which 
$A$ acts. It follows that 
	$t\leq N'$, as the jumps $x_i$ are integers, according to the theorem of 
Hasse-Arf. But the associated gradeds 
	of the filtration $A^{(x)}$ are products of copies of $\Z/p$. It follows 
that every element of $\bar{P}\cap A$ is killed by $p^{N'}$.
\item Finally, as $p\neq \ell$, the elements of the  finite abelian subgroup 
$A\cap \bar{P}\subset \GL_r(\bar{\F}_\ell)$ are simultaneously 
diagonizable matrices with eigenvalues $p^{N'}$-th roots of unity. It follows 
that the cardinality of $A\cap \bar{P}$ is bounded by  $rp^{N'}$, which 
is a constant depending only on $N$ and $r$, as claimed.
\end{enumerate}
\end{proof}

\begin{Remark}\label{explicit-bound}  Our argument yields the upper bound $rp^{N'}J(r)$ for the order of $\overline{P}$. 
Since $N'\leq \ell^n J(r)N\leq \ell r J(r)N$ (since $\ell^{n-1}\leq r$), 
we get an upper  bound for the order of $\overline{P}$ to be $r p^{\ell r J(r)} J(r) N$. 
\end{Remark}

We finish the proof of Proposition~\ref{prop:boundedness}.
 The short exact sequence \eqref{extension} induces a homomorphism of groups 
$\alpha:\Z/M\rightarrow \Aut(\bar{P})$, given by conjugation. 
 Write $H:=\ker(\alpha)$ and let $\bar{I}^\alpha\subset \bar{I}$  be the 
preimage of $H$. By construction $\bar{I}^\alpha\cong \bar{P}\times H$, 
 with both $\bar{P}$ and $H$ characteristic subgroups of $\bar{I}^\alpha$. We 
obtain an injective map $H\hookrightarrow \bar{I}$, such that the 
 composition $H\hookrightarrow \bar{I}\twoheadrightarrow\Z/M$ is the identity on 
$H$, and $H$ is normal in $\bar{I}$. This shows that there is a short exact 
sequence
\[1\rightarrow \bar{P}\rightarrow \bar{I}/H\rightarrow (\Z/M)/H\rightarrow 1.\]
The orders of the two outer terms are bounded by  constants depending only on 
$r, N$ (because of Claim 2 and because $(\Z/M)/H\subset \Aut(\bar{P})$). It 
follows 
that the order of $\bar{I}/H$ is bounded by such a constant. 

Let $L_\rho$ be the Galois extension of $K$ corresponding to $\rho^{-1}(H)$. 
Then the restriction of $\rho$ to $L_{\rho}$ is tame, as it takes values in the 
prime-to-$p$-group $H$. 
On the other hand, by construction we have $[L_{\rho}:K]=[\bar{I}:H]$.
 This finishes the proof.

\end{proof}

\begin{Remark} 
Our argument gives an upper  bound $M$ for the index of the subgroup $H$ to be the order of 
the automorphism  group of $\bar{P}$. Thus, if $M_0=r p^{\ell r J(r)} J(r) N$ is 
our bound for the order of $\bar{P}$, a crude upper bound  for $M$ is the order of the 
permutation group, that is, $M_0!$
\end{Remark}

\begin{Lemma}\label{lemma:pullback}
	Let $K\subset K'$ be an extension of finite Galois extensions of 
complete discretely valued fields with algebraically closed residue fields of 
characteristic $p>0$. 
	If $\rho:\Gal(\bar{K}/K)\rightarrow \GL_r(\F_{\ell^n})$ is a continuous 
representation with $Sw(\rho)$ bounded by $ N$, then $Sw(\rho_{K'})$ is bounded 
by  $[K':K]N$.
\end{Lemma}
\begin{proof}
	Write $G:=\rho(\Gal(\bar{K}/K))$ and $H:=\rho(\Gal(\bar{K}/K'))$. Then 
there is a finite extension $L/K$, such that $G=\Gal(L/K)$ and $H=\Gal(L/K')$. 
	Since the rank $r$ is fixed, and since the Swan conductor is additive, 
we may assume that $\rho$ is irreducible. In this case, the existence of the 
``break decomposition'' 
	implies that there exists a unique $u\in \Z_{>0}$ such that $G_u\neq 
G_{u+1}$. As $H_u=G_u\cap H$, $u$ is also the only index where the ramification 
filtration of $H$ jumps
	(we assume $H\neq 1$).
	One has  $$Sw(\rho)=\phi_{L/K}(u)\cdot r,$$  and
	\[Sw(\rho|_{K'})=\phi_{L/K'}(u)\cdot\dim(V^{H_u}/V^{H_{u-1}})\leq 
\phi_{L/K'}(u)\cdot r,\]
	where $\phi$ denotes the Herbrand function. We know that 
	$$\phi_{L/K}=\phi_{K'/K}\circ \phi_{L/K'},$$ and for any $v\in \Z_{>0}$, 
one has 
	
\[\phi_{K'/K}(v)=\frac{1}{[K':K]}(|\Gal(K'/K)_1|+\ldots+|\Gal(K'/K)_v|)\geq 
\frac{v}{[K':K]} .\]
	It follows that
	
\[\frac{1}{[K':K]}\Swan(\rho|_{K'})\leq\frac{1}{[K':K]}\phi_{L/K'}(u)\cdot r\leq 
\Swan(\rho)\leq N,\]
	as we wanted to prove.
\end{proof}
\section{Global arguments}
Let $k$ be an algebraically closed field, $\bar{X}\subset \P^n_k$ a normal
projective $k$-variety and let $X\subset \bar{X}$ be an open subscheme
such that $(\bar{X}\setminus X)$ is  divisor. Let $D_0$ 
be an irreducible component of $\bar{X}_{\rm reg}\setminus X$, with generic 
point $\xi$. This corresponds to a divisorial valuation $v$ with discrete 
valuation ring ${\mathcal O}_{\bar{X},\eta}$.  We fix $\rho:\pi_1(X,\bar{x})\rightarrow \GL_r(\bar{\F}_{\ell})$, a continuous 
representation with finite image $\bar{I}$ and 
	with ramification bounded by an effective Cartier divisor $\Delta$ which 
has multiplicty at most $N$ along $D_0$. 

	We then introduce some notations.
\begin{itemize}
	\item Let $\pi:Y\rightarrow X$ be finite Galois covering with group
$\bar{I}$ and let $\bar{\pi}:\bar{Y}\rightarrow \bar{X}$ be the
normalization of $\bar{X}$ in $Y$. 
	Denote by  $\eta_1,\ldots, \eta_t$ the codimension $1$
points of $\bar{Y}$ lying over $\xi$. 
\item As $k(X)\subset k(Y)$ is Galois, the
ramification index of the extension of discrete valuation rings
$\mathcal{O}_{\bar{X},\xi}\subset\mathcal{O}_{\bar{Y},\eta_i}$ is
independent of $i$; we denote it by $e$. 
\item Let $f$ be the degree of the residue
extensions $k(\xi)\subset k(\eta_i)$, and $f^{\sep},f^{\insep}$ the separable,
resp.~inseparable degree. Note that $f$ is independent of $i$, and so are
$f^{\sep}$, $f^{\insep}$:  indeed,  the integral closure $B$ of
$\mathcal{O}_{\bar{X},\xi}$ in $k(Y)$ is a Dedekind domain and if
$\mathfrak{P}$ is a prime ideal lying over the maximal ideal
$\mathfrak{m}_{\xi}$, then any Galois automorphism $\sigma\in \bar{I}$
induces an isomorphism of $k(\xi)$-extensions
$B/\mathfrak{P}\xrightarrow{\cong}B/\sigma(\mathfrak{P})$, and $\bar{I}$ acts 
transitively on the set of prime ideals of $B$ lying over $\mathfrak{m}_{\xi}$.
\end{itemize}
\begin{Proposition}\label{prop:curve}
	With $\bar{\pi}:\bar{Y}\rightarrow \bar{X}$ as above, a general complete 
intersection curve $\bar{C}\subset \bar{X}$ has the following properties: 
	\begin{enumerate}
		\item\label{item:sm} $\bar{C}$ is smooth and intersects 
$(\bar{X}\setminus X)_{\red}$ transversely at smooth points.
		\item\label{item:conn}  Write $C:=\bar{C}\cap X$, 
			$D=Y\times_X C$, then $D\rightarrow C$ is Galois \'etale 
with group $\bar{I}$.
		\item\label{item:zhang} Write 
$\bar{D}:=\bar{Y}\times_{\bar{X}}\bar{C}$. The normalization map 
		$\nu:\bar{D}^{\norm}\rightarrow \bar{D}$ is a universal 
homeomorphism.
		\item\label{item:preimage} For $x\in \bar{C}\cap D_0$, there are 
precisely $rf^{\sep}$ points in $\bar{D}$ 
		(and hence in $\bar{D}^{\norm}$) mapping to $x$.
	\end{enumerate}
\end{Proposition}
\begin{proof}  That $\bar{C}$ is nonsingular projective and intersects 
$\bar{X}\setminus X$ only at non-singular points, and that the intersection 
is transverse, follows easily from Bertini's theorem.  This proves 
\ref{item:sm}. For \ref{item:zhang}, we use \cite[Thm.~1.2]{Zha95}
where it is shown that 
the inverse image scheme $\bar{D}$ 
of a general complete intersection curve 
$\bar{C}\subset \bar{X}$   is {\em geometrically unibranch} at all points. In 
addition, it is connected, since $\bar{Y}$ is normal 
(see the proof of  \cite[Cor.~1.7]{Zha95}, for example).  Since $\bar{Y}$ is 
Cohen-Macaulay outside a closed subset of codimension $\geq 3$, we may assume 
$\bar{D}$ is 
contained  in the Cohen-Macaulay locus, and so $\bar{D}$ (which is a complete 
intersection) is also Cohen-Macaulay, and hence reduced. Thus $\bar{D}$ is 
irreducible, and the 
normalization morphism $\bar{D}^{\norm}\to\bar{D}$ is bijective, and a universal 
homeomorphism.  This proves \ref{item:conn}.
For \ref{item:preimage}, note that for every $i\in\{1,\ldots, r\}$, the finite 
map 
 $$(\bar{\{\eta_i\}})_{\red}\rightarrow (\bar{X}\setminus X)_{\red}$$ has degree 
$f=f^{\sep}f^{\insep}$, and can be factored in a homeomorphism of 
  degree $f^{\insep}$ followed by a generically \'etale morphism of degree 
$f^{\sep}$. We may now assume $\bar{C}$ is chosen so that its intersection with 
$D_0$ is in the
  open subset over which the morphism $f^{\sep}$ is \'etale, which again follows 
from  Bertini's theorem. \end{proof}

\begin{Definition} A curve $\bar{C}$ as in Proposition~\ref{prop:curve} is said to be 
 {\em in good position relative to $D_0$}. If $\bar{C}$ is in good position relative to
each irreducible component of $\bar{X}\setminus X$, we say that $C$ is {\em in 
good position relative  to $\bar{X}\setminus X$.}
\end{Definition}

\begin{Theorem}\label{thm:main2}
	We keep the assumptions and notations from above. Let $\bar{C}\subset 
\bar{X}$ be a curve  which is in good position relative to $D_0$. Fix a closed point 
	$x\in \bar{C}\cap D_0$ and a closed point $y\in \bar{D}^{\norm}$ mapping 
to $x$. Denote by $\bar{I}^C_{y/x}\subset \bar{I}$ be 
	the associated decomposition group. Then
	\[|\bar{I}^C_{y/x}| = ef^{\insep}.\]
\end{Theorem}
\begin{proof}
	Let $\nu:\bar{D}^{\norm}\rightarrow \bar{D}$ be the normalization 
morphism.
	We have the following diagram
	\begin{equation*}
		\begin{tikzcd}
			\bar{D}^{\norm}\ar[swap]{dr}{h}\ar{r}{\nu}& 	
\bar{D}\dar{\bar{\pi}_{C}}\rar&\bar{Y}\dar{\bar{\pi}}&\ni \eta_1,\ldots, 
\eta_t\\
			&\bar{C}\rar&\bar{X}&\ni \xi,
		\end{tikzcd}
	\end{equation*}
	and we utilize the notations introduced at the beginning of the section. 
Note that $\bar{I}$ acts on $\bar{D}$ and that the normalization 
	$\nu:\bar{D}^{\norm}\rightarrow \bar{D}$ is an  $\bar{I}$-equivariant 
homeomorphism.

For $x\in \bar{C}\setminus C$, the properties from \Cref{prop:curve} imply that 
$|\bar{\pi}_C^{-1}(x)|=rf^{\sep}$.
 On the other hand, for $y\in \bar{D}^{\norm}$ mapping to $x$ we have
	\[|\bar{\pi}_C^{-1}(x)|=[\bar{I}:\Stab_{\bar{I}}(\nu(y))],\]
	and
	\[\bar{I}^C_{y/x}=\Stab_{\bar{I}}(y)=\Stab_{\bar{I}}(\nu(y)).\]
	It follows that 
	\[|\bar{I}_{y/x}^C|=\frac{|\bar{I}|}{rf^{\sep}}=ef^{\insep}\]
	which is what we wanted to prove.
\end{proof}
Recall that a by \cite[Thm.~1.1]{KS10}, a  Galois cover   is tamely ramified if 
and only if it is in restriction to all curves. We make the small observation, 
implicit in {\it loc.cit.},  that  there are test curves  for all  the generic 
codimension $one$ points at infinity. Explicitly:
\begin{Corollary}
	The covering $\pi:Y\rightarrow X$ is tamely ramified  along the  
codimension one points of   $\bar{X}$ with support in $\bar X\setminus X$ 
	 if and only if $\pi_C:D=Y\times_X C\rightarrow C$ 
	is tamely ramified for a curve  $\bar{C}$  which is in good position 
	relative to $\bar{X}\setminus X$.
\end{Corollary}
\begin{proof}
	If $\pi_C:D\rightarrow C$ is tamely ramified, then for all $x\in 
\bar{C}\setminus C$, and all $y\in \bar{D}^{\norm}$ mapping to $x$, the order of 
$|\bar{I}^C_{y/x}|$ 
	is prime to $p$. The theorem implies that,  for each boundary component $D_0$, 
	$f^{\insep}=1$ and $(e,p)=1$, 
so $\pi$ is tamely ramified with respect $\bar{X}\setminus X$.
\end{proof}


\begin{proof}[Proof of Theorem~\ref{thm:main}]
	Let $\bar{C}\subset \bar{X}$ be a curve as in \Cref{prop:curve}, i.e. which is in good position relative to $D_0$.
	Fix 
$x\in \bar{C}\setminus C$, and $y\in \bar{D}^{\norm}$ mapping to $x$. As 
	$\bar{C}$ intersects $(\bar{X}\setminus X)_{\red}$ transversely, the 
ramification of $\rho|_C$ is bounded by $N\cdot(\bar{C}\setminus C)_{\red}$.  

	Claim 2 in the proof of \Cref{prop:boundedness} showed that the order of 
the $p$-Sylow subgroup of $\bar{I}^C_{y/x}$ is bounded by a constant only 
depending on $N,r,\ell$ (an explicit bound is given by the constant $M_0$). 
As $f^{\insep}$ is a $p$-power, \Cref{thm:main2} implies 
that $f^{\insep}$ is bounded by the same constant.
\end{proof}

\section{Some comments}
 Using the notation from above, for any  $x\in \bar{C}\cap D_0$, and any 
$y\in\bar{D}^{\norm}$ lying 
on $\bar{\{\eta_1\}}$ mapping to $x$,  we have 
\begin{equation}\tag{{$\star$}}\label{eq:discussion}\bar{I}^C_{y/x}
\subset\ker\left(\Stab_{\bar{I}}(\eta_1)\twoheadrightarrow 
\Aut(k(\eta_1)/k(\xi))\right).\end{equation}
The order of $\Stab_{\bar{I}}(\eta_1)$ is $ef$, the order of 
$\Aut(k(\eta_1)/k(\xi))$ is $f^{\sep}$, so the kernel has order $ef^{\insep}$.
It follows that \eqref{eq:discussion} is an equality, and hence that the 
decomposition groups $\bar{I}^C_{y/x}$ only depend on the component 
$\bar{\{\eta_i\}}$ on which 
$y$ lies (recall that $\eta_1,\ldots,\eta_r$ are the generic points of 
$\bar{Y}\setminus Y$).

We can also interpret $\ker\left(\Stab_{\bar{I}}(\eta_1)\twoheadrightarrow 
\Aut(k(\eta_1)/k(\xi))\right)$ as the inertia group of 
$\mathcal{O}^h_{\bar{X},\xi}\subset\mathcal{O}^h_{\bar{Y},\eta_1}$. As such, it 
has a unique $p$-Sylow subgroup, which according to the discussion above, 
has to coincide with the $p$-Sylow subgroup of $\bar{I}^C_{y/x}$ for any $y\in 
\bar{D}^{\norm}$ lying on the closure of $\eta_1$.  

Then this raises the following interesting side question. 
On the $p$-Sylow subgroup of $\bar{I}^C_{y/x}$ we have two filtrations: the 
upper numbering ramification filtration associated to 
$\mathcal{O}_{\bar{C},x}\subset \mathcal{O}_{\bar{D}^{\norm},y}$, for some $y\in 
\bar{D}^{\norm}$ lying on $\bar{\{\eta_1\}}$, and the Abbes-Saito 
ramification filtration on the inertia group of 
$\mathcal{O}^h_{\bar{X},\xi}\subset \mathcal{O}^h_{\bar{Y},\eta_1}$. Do they 
coincide?

\end{document}